\documentclass{article}[9pt]

\usepackage{psfrag}

\usepackage{setspace}
\usepackage[totalwidth=14cm,totalheight=24cm]{geometry}
\usepackage[T1]{fontenc}
\usepackage{epsfig}
\usepackage{amsmath,amssymb,amscd,amsthm}
\usepackage{colortbl}
\usepackage{ae}
\usepackage{color}

\usepackage{textcomp}
\begin{document}

\setcounter{secnumdepth}{3}
\setcounter{tocdepth}{2}

\newtheorem{definition}{Definition}[section]
\newtheorem{lemma}[definition]{Lemma}
\newtheorem{sublemma}[definition]{Sublemma}
\newtheorem{corollary}[definition]{Corollary}
\newtheorem{proposition}[definition]{Proposition}
\newtheorem{theorem}[definition]{Theorem}

\newtheorem{remark}[definition]{Remark}
\newtheorem{example}[definition]{Example}
\newtheorem{question}[definition]{Question}
\newtheorem{conjecture}[definition]{Conjecture}

\newcommand{\cov}{\mathrm{covol}}
\def \tr{{\mathrm{tr}}}
\def \det{{\mathrm{det}\;}}
\def\co{\colon\tanhinspace}
\def\I{{\mathcal I}}
\def\N{{\mathbb N}}
\def\R{{\mathbb R}}
\def\C{{\mathbb C}}
\def\Z{{\mathbb Z}}
\def\Sph{{\mathbb S}}
\def\Tor{{\mathbb T}}
\def\Disk{{\mathbb D}}
\def\Hess{\mathrm{Hess}}
\def\rad{\mathbf{v}}

\def\H{{\mathbb H}}
\def\RP{{\mathbb R}{\mathrm{P}}}
\def\dS{{\mathrm d}{\mathbb{S}}}
\def\Isom{\mathrm{Isom}}

\def\sh{\mathrm{sinh}\,}
\def\ch{\mathrm{cosh}\,}
\newcommand{\arccosh}{\mathop{\mathrm{arccosh}}\nolimits}
\newcommand{\oh}{\overline{h}}

\renewcommand\b[2]{\langle #1,#2\rangle_-}

\newcommand{\mf}{\mathfrak}
\newcommand{\mb}{\mathbb}
\newcommand{\ol}{\overline}
\newcommand{\la}{\langle}
\newcommand{\ra}{\rangle}
\newcommand{\hess}{\mathrm{Hess}\;}
\newcommand{\grad}{\mathrm{grad}}
\newcommand{\M}{\mathrm{MA}}
\newcommand{\II}{\textsc{I\hspace{-0.05 cm}I}}
\renewcommand{\d}{\mathrm{d}}
\newcommand{\A}{\mathrm{A}}
\renewcommand{\L}{\mathcal{L}}
\newcommand{\FB}[1]{{\color{red}#1}}
\newcommand{\note}[1]{{\color{blue}{\small #1}}}
\newcommand{\ap}{\mathrm{area}}
\newcommand{\am}{\mathrm{area}}
\renewcommand{\a}{\mathrm{area}}

\renewcommand{\v}{\mathrm{vol}}

\renewcommand{\thefootnote}{\fnsymbol{footnote}}

%%%%%%%%%%%%%%%%%%%%%%%%%%%%%%%%%%%
%%%%%%%% gere espace texte formule
\setlength{\abovedisplayshortskip}{1pt}
\setlength{\belowdisplayshortskip}{3pt}
\setlength{\abovedisplayskip}{3pt}
\setlength{\belowdisplayskip}{3pt}

%%%%%%%%%%%%%%%%%%%%%%%%%%%%%%%%%%
\title%[Flat metrics on compact surfaces]
{A remark on spaces of flat metrics with cone singularities of constant 
sign curvatures}

\author{Fran\c{c}ois Fillastre and Ivan Izmestiev}
%\address{Universit\'e de Cergy-Pontoise, UMR CNRS 8088, F-95000 Cergy-Pontoise, France}
%\email{francois.fillastre@u-cergy.fr}
\date{v2 \today}

\maketitle


\footnotetext{}     

\setcounter{tocdepth}{3}

\begin{abstract}
By a result of W.~P. Thurston, the moduli space of flat metrics on the sphere 
with $n$ cone singularities of prescribed positive curvatures is a complex 
hyperbolic orbifold of dimension $n-3$. The Hermitian form comes from the area 
of the metric. Using geometry of Euclidean polyhedra, 
we observe that this space has a natural decomposition into real hyperbolic 
convex polyhedra
of dimensions $n-3$ and $\leq \frac{1}{2}(n-1)$.

By a result of W.~Veech, the moduli space of flat metrics on a compact surface 
with cone singularities of prescribed negative curvatures has a foliation whose leaves have a local structure of complex 
pseudo-spheres. The complex structure comes again from the area of the metric.
The form can be degenerate; its signature depends on the curvatures prescribed.
Using polyhedral surfaces in Minkowski space, we show that this moduli space 
has a natural decomposition into spherical convex polyhedra.

\end{abstract}

\tableofcontents
%%%%%%%%%%%%%%%%%%%%%%%%%%%%%%%%%%%%%%%%%%%%%%%%%%%%%%%%%%

\section{Flat metrics}
Consider a collection of Euclidean triangles glued isometrically along some pairs of edges
so that the resulting topological space is homeomorphic to a closed orientable surface $S$.
This puts on $S$ a path metric $m$ which is a \emph{flat metric with conical singularities}.
The metric $m$ can also be described by means of an atlas whose targets are open subsets
of the Euclidean plane and of the Euclidean cones and whose transition maps are local Euclidean isometries.
The two definitions are equivalent, because any flat metric with conical singularities has a geodesic triangulation.
Note that a \emph{geometric flip} as on Figure~\ref{flip} changes the triangulation but preserves the metric.
We refer to \cite{troyanov2,bon} for more details.

Let $p_i$ be a vertex of the triangulation and let $\alpha_i$ be the sum of the angles around $p_i$.
If $\alpha_i \ne 2\pi$, then $p_i$ is called a \emph{cone singularity} with the \emph{cone angle} $\alpha_i$.
The \emph{(singular) curvature} at $p_i$ is
\[
k_i=2\pi - \alpha_i~.
\]
See Figure~\ref{courbure} and Figure~\ref{ex met plates} for examples. 

\begin{figure}
\begin{center}
\psfrag{a}{$\textcolor{blue}{\alpha < 2\pi}$}
\psfrag{kp}{$\textcolor{red}{k>0}$}
\psfrag{ag}{$\textcolor{blue}{\alpha > 2\pi}$}
\psfrag{kn}{$\textcolor{red}{k<0}$}
\includegraphics[width=0.7\linewidth]{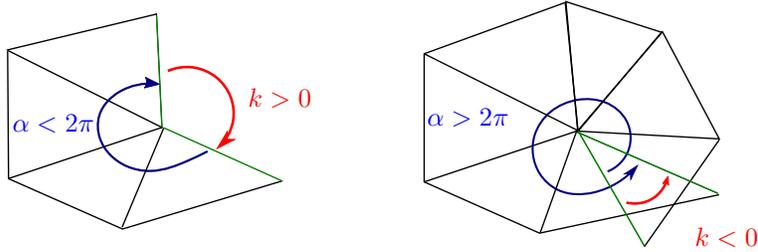}\caption{A vertex of positive curvature and a vecter of negative curvature.}\label{courbure}
\end{center}
\end{figure}

Throughout the paper, $n$ denotes the number of cone singularities of $(S, m)$.

\begin{lemma}[Discrete Gauss--Bonnet formula]
For every flat metric with cone singularities on a surface of genus $g$, we have
\begin{equation}\label{GB}
\sum_{i=1}^n k_i = 2\pi(2-2g)~.
\end{equation}
\end{lemma}
\begin{proof}
Let $T$ the number of triangles and $E$ the number of edges in some geodesic triangulation of $(S, m)$.
We have $E=\frac{3}{2}T$, so that the Euler formula
\[
T-E+n =  2-2g
\]
implies
\[
2n-T=2 (2-2g)~.
\]
On the other hand, the sum of all the angles of all the triangles is equal to the sum of all the cone angles, hence
\[
\sum_{i=1}^n k_i = 2\pi n - \sum_{i=1}^n\alpha_i = 2\pi n - \pi T = 2\pi(2-2g)~.
\]
\end{proof}

\begin{figure}[ht]
\begin{center}
\psfrag{t1}{$\theta_1$}
\psfrag{t2}{$\theta_2$}
\psfrag{t3}{$\theta_3$}
\psfrag{2t1}{$2\theta_1$}
\psfrag{2t2}{$2\theta_2$}
\psfrag{2t3}{$2\theta_3$}
\psfrag{2p}{$2 \pi$}
\psfrag{3p}{$3 \pi$}
\psfrag{3p4}{$\frac{3\pi}{4}$}
\includegraphics[width=0.5\linewidth]{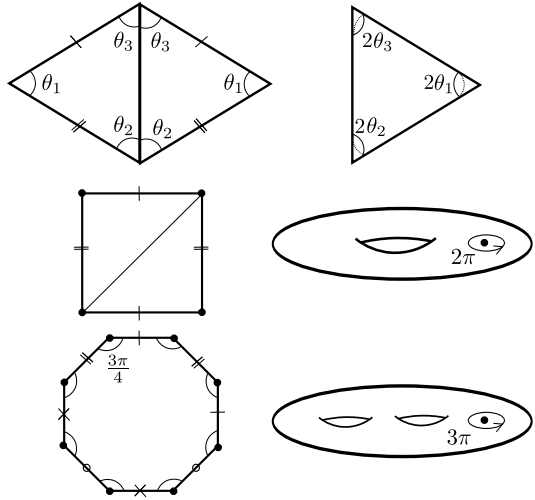}
\end{center}
\caption{Examples of flat metrics for surfaces of different genus.}\label{ex met plates}
\end{figure}

For any collection of cone angles
\[
\alpha=(\alpha_1,\ldots,\alpha_n) \in \left((0,2\pi) \cup (2\pi, +\infty)\right)^n
\]
let $\mathcal{M}(S,\alpha)$ be the set of all flat metrics on $S$ with the cone angles $\alpha$
modulo orientation-preserving similarities isotopic to the identity.
(In particular, the similarity has to send $p_i$ to $p_i$ for all $i$.)

By a theorem of M.~Troyanov,  $\mathcal{M}(S,\alpha)$  is not empty if and only if $\alpha$ satisfies the discrete Gauss--Bonnet formula.
In this case, $\mathcal{M}(S,\alpha)$ has a structure of a complex manifold of (complex) dimension $3g-3+n$ \cite{troyanov1,troyanov2}.

\begin{figure}[ht]
\begin{center}
\includegraphics[width=0.4\linewidth]{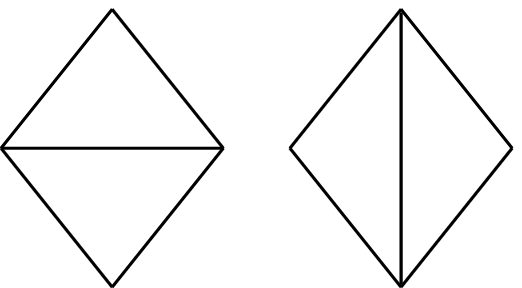}
\end{center}
\caption{Flip.}\label{flip}
\end{figure}

\subsection{Positive curvatures}

Here we consider the case of $\alpha \in (0,2\pi)^n$, that is of $k_i > 0$ for all $i$.
The discrete Gauss-Bonnet formula implies that a flat metric with positive curvatures can exist only on the sphere.
Besides, the curvatures must satisfy
\[
\sum_{i=1}^n \kappa_i = 4\pi~.
\]
Since $\kappa_i < 2\pi$, we have $n \ge 3$.

A convex polytope in $\R^3$ is the convex hull of finitely many points whose affine hull is all of $\R^3$.
The intrinsic metric on the boundary of a convex polytope is a flat metric with positive curvatures.
Vice versa, a famous theorem of A.~D.~Alexandrov \cite{alex2} says that
any flat metric with positive curvatures on the sphere can be obtained in this way (if we allow the polytope to degenerate to a convex polygon).

\begin{theorem}[A.~D.~Alexandrov]
\label{thm:Alex}
Let $m$ be a flat metric with positive curvature on the sphere.
Then either $(\Sph^2, m)$ is isometric to a doubly-covered convex polygon (see Figure \ref{ex met plates}, top),
or there is a unique up to isometry convex polytope whose boundary is isometric to $(\Sph^2, m)$.
\end{theorem}

Recall a celebrated construction of Thurston \cite{th,schwartz}, \cite{kojima}.
Cut $(\Sph^2,m)$ along a geodesic tree connecting the cone points.
The result is a disk with a flat metric and piecewise geodesic boundary.
The disk can be locally isometrically immersed in the complex plane as a polygon with possible self-intersections.
This construction is similar to unfolding a convex polytope.
On the problem of a non-self-intersecting unfolding see \cite{BG17}.

By the Euler formula, the disk has $2(n-1)$ sides. Each side gives a vector of the plane.
Identify $\R^2$ with $\C$, so that to represent the vectors by complex numbers.
One can prove that $n-2$ of these complex parameters suffice to recover the polygon, and together with it the metric on the sphere.
Since the metrics are considered up to similarities,
and a similarity corresponds to multiplication by a non-zero complex number,
this yields local charts from  
$\mathcal{M}(S,\alpha)$ to $\mathbb{C}\mathbb{P}^{n-3}$.
We will call the corresponding set in $\mathbb{C}^{n-2}$ an  \emph{unfolding chart}.

The \emph{area} of $m$ is the sum of the area of the triangles. 
In an  unfolding chart, the area extends to a Hermitian form on 
   $\mathbb{C}^{n-2}$. 
By induction on $n$, using a procedure of chasing the curvature,
Thurston proved that this form has signature $(1,n-3)$.
Since the image of the unfolding chart lies in the positive cone of the area quadratic form, 
the space $\mathcal{M}(S,\alpha)$ has a natural complex hyperbolic structure, and
\[
\dim_{\C} \mathcal{M}(S,\alpha) = n-3~.
\]

On the other hand, the metric $m$ can be realized as the double of a convex polygon or as the boundary of a convex polytope,
see Theorem \ref{thm:Alex}.
Consider two cases.

\textbf{Case 1.} $(\Sph^2, m)$ is isometric to a double of a convex polygon $P$.
Let $H(P)$ be the set of convex polygons, modulo translations, with edges parallel to those of $P$.
Such a polygon is determined by $n-2$ real parameters, and one can identify $H(P)$ with a convex polyhedral cone in $\R^{n-2}$.
Let $H_1(P)$ be the set of polygons from $H(P)$ of area one. The area of the polygons extends to a quadratic form on $\R^{n-2}$.
C.~Bavard and \'E.~Ghys \cite{BG} have observed that this quadratic form has signature $(1,n-3)$.
This makes $H_1(P)$ to a convex hyperbolic polyhedron, and we have
\[
\dim_{\R} H_1(P) = n-3~.
\]
See Section~\ref{polygons} for more details.

\textbf{Case 2.} $(\Sph^2, m)$ is isometric to the boundary of a convex polytope $P$.
Denote by $H(P)$ the set of convex polytopes, modulo translations, with faces parallel to those of
$P$ and with the same combinatorics as $P$.
(That is, if two faces of $P$ intersect along an edge,
then the two corresponding faces of every $Q \in H(P)$ also intersect along an edge.)
The set $H(P)$ has a structure of a $d$-dimensional convex polyhedral cone.
The dimension $d$ is maximal when $P$ is  \emph{simple},
i.~e. every vertex is incident to exactly three edges.
In this case, $d=f-3$, where $f$ is the number of faces of $P$.
By the Euler formula, if $n$ is the number of vertices, then $f=n/2+2$, thus $d=n/2-1$.
The minimum possible dimension $d$ is $1$.
This is attained if all $Q \in H(P)$ are homothetic to $P$, for example if $P$ is an octahedron.

Let $H_1(P)$ be the set of $Q \in H(P)$  with $\a(\partial Q) = 1$.  
The area of the boundary extends to a quadratic form on $\R^d$.  
As was observed in \cite{FI}, this quadratic form has signature
$(1,d-1)$. In particular, $H_1(P)$ is a convex hyperbolic polyhedron, and we have
\[
0 \le \dim_\R H_1(P) \le \frac{n}{2} - 2~.
\]
See Section~\ref{area polytope} for more details.

In each of the above cases there is an injective linear map from $\R^d$ to $\mathbb{C}^{n-2}$
such that the image of $H(P)$ is contained in an unfolding chart of $\mathcal{M}(S,\alpha)$.
This map preserves the area quadratic form.
The difference in the dimensions of $\mathcal{M}(S,\alpha)$ and $H_1(P)$ is explained by the fact that
in an unfolding chart the parameters of $H(P)$ allow to change only the lengths of the edges of the development, but not their directions.

\subsection{Negative curvatures}
Thurston's construction was generalized to the space of flat metrics with the curvatures of any signs by W.~Veech 
\cite{veech,manh,GP}.

A flat metric $m \in \mathcal{M}(S,\alpha)$ can be cut along a geodesic tree connecting the cone points.
This gives a disk with a flat metric and a piecewise geodesic boundary consisting of $2(2g-1+n)$ edges.
The manifold $\mathcal{M}(S,\alpha)$ is foliated by submanifolds $\mathcal{M}(S,\alpha, h)$ (called \emph{Veech leaves}),
where $h$ denotes the rotational part of the holonomy of the metric.
For $S = \Sph^2$, the rotational part of the holonomy is determined by the cone angles, for surfaces of higher genus it is not.
It turns out that for a fixed rotational part of the holonomy the shape of the disk (and hence the metric $m$)
can be recovered from $k$ complex parameters, where
\[
k=
\begin{cases}
2g-2+n,& \text{ if there is a cone angle which is not an integer multiple of }2\pi;\\
2g-1+n,& \text{ if all cone angles are integer multiples of }2\pi.
\end{cases}
\]
This yields local charts from $\mathcal{M}(S,\alpha,h)$ into $\mathbb{CP}^{k-1}$.
In particular, we have
\[
\dim_{\mathbb{C}}(\mathcal{M}(S,\alpha, h)) = 2g+n-3\, \mbox{ or }\, 2g+n-2~.
\]

As before, the area of a surface is a Hermitian form on $\mathbb{C}^k$ in the coordinates of an unfolding chart.
The signature of this quadratic form depends on the cone angles $\alpha$ in a complicated way described in Theorem~14.6 of \cite{veech}.
Table~\ref{tab sign} shows some examples in genus $2$ with all curvatures negative.

The following are some general facts about the signature of the area.
\begin{itemize}
\item If all cone angles are integer multiples of $2\pi$, then the signature is $(g-1,g-1)$.
In particular, there is an $(n+1)$-dimensional kernel.
\item If none of the cone angle is an integer multiple of $2\pi$, then the area quadratic form is non-degenerate.
\end{itemize}

\begin{table}
\begin{center}
\begin{tabular}{c|c|c}
$(a_1,\ldots,a_n)$ & signature & dimension \\
 \hline
 $(1,1)$   & $(1,1)$  & $5$ \\
$(1/2,1/2,1/2,1/2)$    & $(3,3)$  & $6$ \\
$(1/3,1/3,1/3,1/3,1/3,1/3)$ & $(5,3)$ & $8$ \\
$(4/3,1/3,1/3)$ & $(3,2)$ & $5$\\
$(1,1/2,1/2)$ & $(2,2)$ & $5$ \\
$(1,1/3,1/3,1/3)$ & $(3,2)$& $6$\\
\end{tabular}
\caption{The signature of the area form on a surface of genus $2$
in some special cases. Here $k_i=-2\pi a_i$. }\label{tab sign}
\end{center}
\end{table}

Let us now consider the case when $\alpha \in (2\pi, +\infty)^n$, that is $k_i < 0$ for all $i$.
The curvatures can satisfy the Gauss--Bonnet formula only for a surface of genus $g > 1$.

The following is an analog of Theorem \ref{thm:Alex}, see \cite{fillastre,leo}.
\begin{theorem}
\label{thm:FBr}
For any flat metric with negative curvatures on a closed orientable surface $S$
there is a convex  polyhedron $P$ in the Minkowski space $\R^{2,1}$
invariant under the action of a cocompact lattice $\Gamma$ of $SO(2,1)$
such that the boundary of $P/\Gamma$ is isometric to $(S, m)$. 
Moreover, $P$ is unique up to composition by Lorentzian linear isometries.
\end{theorem}
Let us call $P$ a \emph{convex $\Gamma$-polyhedron}, see Section~\ref{fuchsian} for precise definitions.
There are two major differences from the Euclidean case.
First, there is no analog of the double of a convex polygon in Theorem \ref{thm:FBr}.
Second, the set of convex $\Gamma$-polyhedra is not invariant under translations.

Denote by $H^\Gamma(P)$ the set of convex  $\Gamma$-polyhedra with faces parallel to those of
$P$ and with the same combinatorics as $P$. 
The set $H^\Gamma(P)$ is a $d$-dimensional convex cone.
The dimension $d$ is maximal when $P$ is simple.
In this case, $d=f$, where $f$ is the number of orbits of faces of $P$ under the action of $\Gamma$.
By the Euler formula, if $n$ is the number of vertices, then $f=n/2+2-2g$. 
The minimum possible dimension is $d=1$, attained when all polyhedra in $H^\Gamma(P)$ are homothetic of $P$.

The area of the boundary of $Q \in H^\Gamma(P)$ is positive
(the restriction of the Minkowski scalar product to $\partial Q$ is sign definite).
Let $H^\Gamma_1(P)$ be the set of $Q \in H^\Gamma(P)$  with $\a(\partial Q) = 1$.
The area 
extends to a quadratic form on $\R^d$. 
It turns out that the area is \emph{positive definite}, so that 
$H_1^\Gamma(P)$ is a convex \emph{spherical} polyhedron with
\[
0 \le \dim_\R(H_1^\Gamma(P)) \le n/2+1-2g~.
\]
See Section~\ref{fuchsian} for more details.

There is an injective linear map from $\R^d$ to $\mathbb{C}^k$
such that the image of $H^\Gamma(P)$ is contained in an unfolding chart of a Veech leaf $\mathcal{M}(S,\alpha, h)$.
This linear map preserves the area quadratic form.

\section{Spaces of polygons}\label{polygons}

General references for mixed volume and mixed area are
\cite{schneider,alex1,alex2}. See also \cite[Appendix A]{Izm10}.

Let $P$ be a convex $n$-gon.
Introduce the following notation:

$$\tilde{H}(P)=\{\mbox{convex polygons with edges parallel to }P\} $$
$$H(P)=\tilde{H}(P)/\mbox{translations}$$
$$H_1(P)=H(P)/\mbox{ homotheties}~. $$

Number the sides of $P$ arbitrarily by the numbers from $1$ to $n$ and use the same numbering for the sides of every $Q \in \tilde{H}(P)$.
Let $h_i$ be the support number of the $i$-th side of $Q$, that is the distance from the origin to the the line containing the side. The distance is calculated with a sign: it is positive if the line spanned by the side does not separate the polygon from the origin.
The support numbers determine a polygon in $\tilde{H}(P)$ uniquely, so that we can identify $Q$ with a vector $h \in \R^n$.

Let $\ell_i$ be the length of the $i$-th side of $Q$. If the adjacent sides have the indices $j$ and $k$, then we have (see Figure~\ref{fig sup edge}):

\begin{equation}\label{length edge}\ell_i(h)=\frac{1}{\sin \theta_{j}}\left( h_j + h_i \cos \theta_{j}\right)+\frac{1}{\sin \theta_{k}}\left( h_k + h_i \cos \theta_{k}\right)~.
\end{equation}

\begin{figure}[ht]
\begin{center}
\psfrag{p}{$p$}
\psfrag{0}{$0$}
\psfrag{ui}{$u_i$}
\psfrag{uj}{$u_j$}
\psfrag{hi}{$h_i$}
\psfrag{hj}{$h_j$}
\psfrag{vij}{$v_{ij}$}
\psfrag{ti}{$\theta_j$}
\psfrag{ptj}{$\pi-\theta_j$}
\psfrag{lij}{$\ell_{ij}$}
\psfrag{lik}{$\ell_{ik}$}
\includegraphics[width=0.4\linewidth]{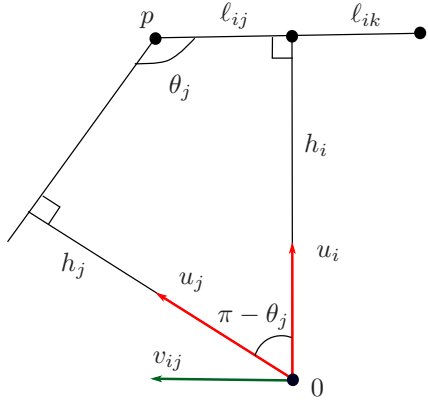}
\end{center}
\caption{Length $\ell_i$: 
$p=\ell_{ij}v_{ij}+h_iu_i$, $u_j=\cos (\pi-\theta_{j}) u_i + \sin (\pi- \theta_{j}) v_{ij}$.
As $h_j=\langle u_j,p\rangle$, we obtain
$\ell_{ij}=\frac{1}{\sin \theta_{j}}\left( h_j + h_i \cos \theta_{j}\right) $
and $\ell_i=\ell_{ij}+\ell_{ik}$. }\label{fig sup edge}
\end{figure}

The set $\tilde{H}(P) \subset \R^n$ is the solution set of a system of linear inequalities $\ell_i(h) > 0$, which is an open convex polyhedral cone.

By decomposing $Q$ into triangles with the common vertex $0$ we compute the area of $Q$:
$$\a(Q)=\frac{1}{2}\sum_{i=1}^n h_i \ell_i(h)~.$$
Since every $\ell_i$ is a linear function of $h \in \R^n$, the area is the restriction of a quadratic form on $\R^n$
(where $Q$ is identified with the vector of its support numbers $h \in \R^n$).
Denote this quadratic form $\a_P$. This form is determined by the angles between the sides of the polygon $P$, see the above formulas.
The \emph{mixed area} is the polarization of this quadratic form. That is, if $R \in \tilde{H}(P)$ is a polygon with the support numbers $k \in \R^n$, then we put

\begin{equation}\label{def mixed area pol}
\a_P(h,k)=\frac{1}{4}(\sum h_i\ell_i(k) +k_i\ell_i(h) )~.
\end{equation}

\begin{figure}[ht]
\begin{center}
\psfrag{e}{$\epsilon$}
\psfrag{a}{$\alpha$}
\psfrag{aa}{$\alpha'$}
\psfrag{l}{$\ell_i$}
\psfrag{h}{$h$}
\includegraphics[width=0.5\linewidth]{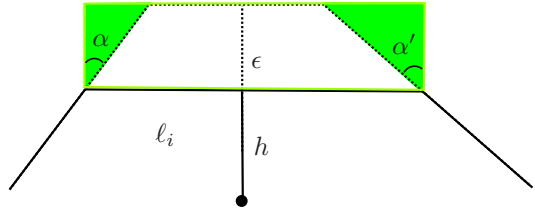}
\end{center}
\caption{$\partial_i \a_P=\ell_i$. The area 
 of the yellow rectangle minus the areas of the green triangles is $\epsilon \ell_i - \frac{1}{2}\epsilon^2(\tan \alpha + \tan \alpha')$.}
 \label{der aire}
\end{figure}

%
%
%\begin{figure}[ht]
%\begin{center}
%\includegraphics[scale=1.5]{deriveeaire.png}
%\end{center}
%\caption{$\partial_i \a_P(P)=\ell_i$. The area 
% of the yellow rectangle minus the areas of the green triangles is $\epsilon \ell_i - \frac{1}{2}\epsilon^2(\tan \alpha + \tan \alpha')$.}
% \label{der aire}
%\end{figure}

The following lemma shows that the symmetrization in the above formula is redundant.
\begin{lemma}
\label{lem:MixAreaFormula}
The mixed area can be computed by the formula
$$\a_P(h,k)=\frac{1}{2}\sum_{i=1}^n h_i \ell_i(k)~.   $$
\end{lemma}

\begin{proof}

The fundamental remark is that $\partial_i \a_P (h)=\ell_i(h)$, where $\partial_i$ is the partial differentiation with respect to $h_i$,
see Figure~\ref{der aire}.
Thus we have
\[
\a_P(h,k) = \frac12 D_h\a_P (k) = \frac12 \sum_i h_i \partial_i\a_P(k) = \frac12 \sum_i h_i\ell_i(k)
\]
See \cite{FBG} for another 
argument.
\end{proof}

The signature of the area quadratic form was computed by Bavard and Ghys in \cite{BG} using the induction on $n$.
Below we present an argument based on the theory of mixed areas.

For a point $x\in \R^2$, let $h^x \in \R^n$ be the vector with components
$h^x_i=\langle x,n_i\rangle$. This is the support vector of $\{x\}$, with $\{x\}$ viewed as
a degenerate convex polygon from $\tilde{H}(P)$.

What can we say about the signature of $\a_P$?

\begin{enumerate}
\item \textbf{There is at least one positive direction}. This is simply because for any
$h\in \tilde{H}(P)$, $\a_P(h)>0$.
\item \textbf{The dimension of its kernel is at least two}. Indeed, let $x\in \R^2$.
Then $Q+\{x\}$ is the translation of $Q$ by the vector $x$, and
$$0=\a_P(Q+\{x\})-\a_P(Q)=2\a_P(Q,\{x\})~,$$
 so for any $h\in \tilde{H}(P)$, 
$\a_P(h,h^x)=0$, and as $\tilde{H}(P)$ spans $\R^n$, $h_x\in \mathrm{Ker}(\a_P)$.
\end{enumerate}

In order to be able to say more, we need to consider the 
bigger set  $\mathcal{K}^2_0$ of all convex figures in the plane (i.e. compact convex sets with non empty interior).
This set is invariant under \emph{Minkowski addition}
$$K_1+K_2=\{x+y \mid x\in K_1, y\in K_2 \} $$
and (positive) homotheties: for $\lambda >0$,
$$\lambda K = \{ \lambda x \mid x\in K \}~. $$

Note that Minkowski addition and homotheties correspond to the linear operations with the support vectors.

\begin{theorem}[Minkowski inequality]
For any $K_1,K_2\in \mathcal{K}_0^2$,
$$\a_P(K_1,K_2)^2\geq \a_P(K_1)\a_P(K_2)~.$$
\end{theorem}

 From this inequality, we can deduce the following information about 
 $\a_P$.

\begin{enumerate}  \setcounter{enumi}{2}
\item \textbf{The positive index of $\a_P$ is one}. 
Indeed, Minkowski inequality says that, for  $h,k\in \tilde{H}(P)$,
$$\a_P(h,k)^2\geq \a_P(h)\a_P(k)~.$$
Suppose that the positive index of $\a_P$ is larger than one.
Then, for $h\in \tilde{H}(P)$, there exists
$v\in \R^n$ such that $\a_P(h,v) = 0$ and $\a_P(v,v)>0$. But as  $\tilde{H}(P)$ is open, there exists $\epsilon>0$ such that
$h+\epsilon v\in \tilde{H}(P)$. So $\a_P$ is positive definite on the plane spaned by 
$v$ and $h$. But then Cauchy--Schwarz inequality applies, and contradicts Minkowksi inequality.
\end{enumerate}

Also, we need the description of the equality case in Minkowski inequality.

\begin{theorem}[Equality case] 
Equality occurs in Minkowski inequality if and only if there is $x\in \R^2$ and $\lambda >0$ with $K_1=\{x\}+\lambda K_2$.
\end{theorem}

\begin{enumerate}  \setcounter{enumi}{3}
\item \textbf{The dimension of the kernel of $\a_P$ is 2}.
Indeed, the description of the equality case in Minkowski inequality says that 
 for $h,k\in  \tilde{H}(P)$,
$\a_P(h,k)= \a_P(h)\a_P(k)~$ if and only if there is $x\in \R^2$ and 
$\lambda >0$ with $h=h^x+\lambda k$.

Let $v$ be in the kernel. There exist $h,k\in \tilde{H}(P)$ such that
$v=h-k$. From $\a_P(h,v)=0$ it follows that $\a_P(h,k)=\a_P(h)$, and from
$\a_P(k,v)=0$, that $\a_P(h,k)=\a_P(k)$. The Minkowski inequality thus holds as equality,
and there is $x\in \R^2$ with $h=h^x+\lambda k$ and $v=h^x+(\lambda-1)k$. From
$\a_P(v,Q)=0$ we deduce that $\lambda=1$ so $v=h^{x}$.
\end{enumerate}

\begin{corollary}
The bilinear form $\a_P$ has signature $(1,2,n-3)$ on $\R^n$.
\end{corollary}

\begin{corollary}\label{cor embed polygons}
Let $P$ be a convex polygon with angles $\alpha_1 /2,\ldots,\alpha_n /2$, and let $\alpha=(\alpha_1,\ldots,\alpha_n)$.
Then  $H_1(P)$ is a convex polyhedron in $\H^{n-3}$, which isometrically embeds into $\mathcal{M}(\Sph^2, \alpha)$.
\end{corollary}
\begin{proof}
Let $Q\in\tilde{H}(P)$. Consider another copy of $Q$, glued to $Q$ along an edge.
This corresponds to the image of the metric on the sphere given by the doubling of $Q$ for an unfolding chart.
The parameters  are the complex numbers given by the edges seen as vectors.
In $\tilde{H}(P)$, the angles between the edges are fixed, and one can change the length of the edges,
which are linear functions of the support vectors.
\end{proof}

\begin{remark}\label{remark polygon}{\rm
Here is some more information about $H_1(P)$. First, $H_1(P)$ is a simple convex polyhedron. 
It always has finite volume; it is compact if and only if $P$ has no parallel edges. Actually  $H_1(P)$ is an orthoscheme
(roughly speaking, each facet meets non-orthogonally at most two other facets),
and its dihedral angles can be easily computed from the angles of $P$.
In particular, one can formulate under what conditions on $P$ the polyhedron $H_1(P)$ is a Coxeter orthoscheme.
Moreover, if $P'$ has the same angles as $P$, but their order differs by the transposition of two consecutive angles,
then $H_1(P)$ and $H_1(P')$ can be glued along 
isometric facets. Performing a gluing for all permutations of angles, one obtains a hyperbolic cone-manifold, which embeds isometrically into $\mathcal{M}(\Sph^2, \alpha)$, see \cite{BG,FBG} for more details.

These properties fail when $P$ is replaced by a convex polyhedron, see the next section.
}
\end{remark}

\begin{remark}\label{remark debin}{\rm
The space $H_1(P)$ naturally carries a hyperbolic metric due to the Minkowski inequality. 
But this inequality holds for all convex figures, not only for polygons. 
Thus, in a similar way it follows that the set of convex figures, up to translations and homotheties, is a convex subset of 
an infinite dimensional hyperbolic space. This is the subject of the last chapter of C. Debin PhD thesis \cite{debin}. 
See also Section~\ref{sobolev}.
 }
\end{remark}

\section{Boundary area of a convex polytope}\label{area polytope}

Let $P$ be a convex polytope in $\R^3$ with $m$ faces.
Introduce the following notation:

\begin{multline*}
\tilde{H}(P)=\{\text{convex polytopes with faces parallel to those of }P\\
\text{ and the same combinatorics as }P\}
\end{multline*}
\[H(P)=\tilde{H}(P)/\mbox{translations}\]
\[H_1(P)=H(P)/\mbox{homotheties}~. \]

The \emph{Gauss image} or the normal fan of $P$ is a tesselation of $\Sph^2$ by convex spherical polygons.
The vertices of the polygons are the outward unit normals to the faces of $P$.
There is an edge between two vertices if and only if the corresponding faces have a common edge.
Thus the polygons of the Gauss image of $P$ correspond to the vertices of $P$.
The angles of the Gauss image of a vertex are related to the face angles at that vertex, see Figure~\ref{fig gauss}.
We have $Q \in \tilde{H}(P)$ if and only if $Q$ has the same Gauss image as $P$.

\begin{figure}[ht]
\begin{center}
\psfrag{C}{$C$}
\psfrag{p}{$p$}
\psfrag{S}{$S$}
\psfrag{b}{$\beta_1$}
\psfrag{bb}{$\pi-\beta_1$}
\includegraphics[width=0.5\linewidth]{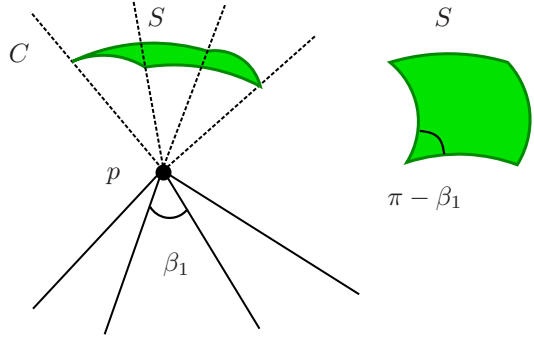}
\end{center}
\caption{The Gauss image of a vertex $p$ of a convex polyhedron $P$ is the intersection of a cone $C$,
whose edges are orthogonal to the faces of $P$ at $p$, with a unit sphere centered at $p$.
This gives a spherical convex polygon $S$, whose interior angles are the interior dihedral angles of $C$,
equal to $\pi$ minus the face angles of $P$. By the Gauss--Bonnet formula applied to the spherical polygon $S$,
the cone angle around $p$ (the sum of the face angles) is equal to $2\pi - $area$(S) < 2\pi$.}\label{fig gauss}
\end{figure}

Observe that
\begin{itemize}
\item a polytope $Q \in \tilde{H}(P)$ is determined by the (signed) distances of its faces from the origin, called the support numbers of $Q$;
\item polytopes with the same Gauss image have the same dihedral angles at the corresponding edges and the same face angles
at the corresponding vertices;
\item as a consequence, the induced metrics on the boundaries of polytopes with the same Gauss image
have the same singular curvatures at the corresponding vertices.
\end{itemize}

%
%
%\begin{figure}[ht]
%\begin{center}
%\includegraphics[scale=1]{gauss.png}
%\end{center}
%\caption{The Gauss image of a vertex $p$ of a convex polyhedron $P$ is the intersection of a cone $C$, whose edges are orthogonal to the faces of   $P$ at $p$, with a unit sphere centered at $p$. This gives a spherical convex polygon $S$, whose interior angles are the interior dihedral angles of $C$, which correspond to $\pi$ minus the face angles of $P$. By the Gauss--Bonnet formula applied to the spherical polygon $S$, the cone angle around $p$ (the sum of the face angles) is equal to $2\pi - $area$(S) < 2\pi$.}\label{fig gauss}
%\end{figure}

Identifying a polytope $Q \in \tilde{H}(P)$ with its support numbers, we have $\tilde{H}(P) \subset \R^m$.

Number the faces of $P$ arbitrarily by the numbers from $1$ to $m$ and use the same numbering for all $Q \in \tilde{H}(P)$.
Let $F_i(Q)$ be the $i$-th face of $Q$.
In the plane spanned by $F_i(Q)$ choose as the origin the orthogonal projection of the origin of $\R^3$.
This allows to define the support numbers $h_{ij}$ of $F_i(Q)$, where the index $j$ ranges over the faces adjacent to $F_i$.
The support numbers of a face can be computed from those of the polytope, see Figure~\ref{fig supface}.

\begin{figure}[ht]
\begin{center}
\psfrag{hi}{$h_i$}
\psfrag{h}{$h_j$}
\psfrag{hij}{$hij$}
\psfrag{phi}{$\varphi_{ij}$}
\includegraphics[width=0.5\linewidth]{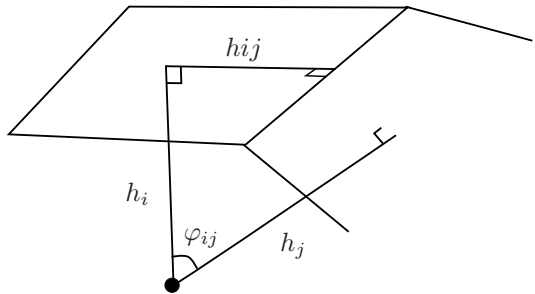}
\end{center}
\caption{From Figure~\ref{fig sup edge}, 
$h_{ij}=\frac{1}{\sin \varphi_{ij}}(h_j-h_i\cos\varphi_{ij})$.}\label{fig supface}
\end{figure}

%
%
%
%\begin{figure}[ht]
%\begin{center}
%\includegraphics[scale=1]{supface.png}
%\end{center}
%\caption{From Figure~\ref{fig sup edge}, 
%$h_{ij}=\frac{1}{\sin \varphi_{ij}}(h_j-h_i\cos\varphi_{ij})$.}\label{fig supface}
%\end{figure}
%

Applying the formula on Figure~\ref{fig sup edge} twice,
we see that the edge lengths of a polytope from $\tilde{H}(P)$ are linear functions $\ell_{ij}(h)$ of the support numbers of the polytope.
A vector $h \in \R^m$ belongs to $\tilde{H}(P)$  if and only if $\ell_{ij}(h) > 0$ for all $i,j$.
Therefore $\tilde{H}(P)$ is a relatively open convex polyhedral cone in $\R^f$.
A polytope is called \emph{simple} if each vertex is incident to exactly three edges.
The cone $\tilde{H}(P)$  is full-dimensional (and then open in $\R^m$) if and only if the polytope $P$ is simple.

The volume of $Q \in \tilde{H}(P)$ can be computed by the formula
\[
\v(Q)=\frac{1}{3}\sum_{i=1}^m h_i(Q) \a(F_i(Q))~.
\]
In terms of the support numbers, this is a homogeneous polynomial of degree $3$ in $m$ variables:
\[
\v_P(h)=\frac{1}{3}\sum_{i=1}^m h_i \a_{F_i}(h_{i \bullet})~,
\]
where $h_{i \bullet}$ is the vector with components $h_{ij}$, $j$ ranging over the faces adjacent to $F_i$.
We define the mixed volume as a symmetric $3$-linear form which is the polarization of $\v_P$.
Similarly to Lemma \ref{lem:MixAreaFormula} one can prove that
\[
\v_P(h,k,p)=\frac{1}{3}\sum_{i=1}^m h_i \a_{F_i}(k_{i \bullet}, p_{i \bullet})~.
\]

\begin{theorem}[Alexandrov--Fenchel theorem]\label{thm:AF1 simple}
Let $P$ be a simple convex polytope. For any $p\in \tilde{H}(P)$ and any $h, k \in \R^m$ we have
$$\v_P(h,k,p)^2\geq \v_P(h,h,p)\v_P(k,k,p)~, $$
and equality occurs if and only if there exist $x\in \R^3$ and $\lambda\in \R$ with 
$h=h^x + \lambda k$.
\end{theorem}

The area of the boundary of $Q \in \tilde{H}(P)$ is the restriction to $\tilde{H}(P) \subset \R^m$ of a quadratic form on $\R^m$
\[
\ap_P(h)=\sum_{i=1}^m \a_{F_i}(h_{i \bullet}).
\]
Its polarization is a symmetric bilinear form
\[
\ap_P(h,k)=\sum_{i=1}^m \a_{F_i}(h_{i\bullet}, k_{i\bullet})~.
\]
Denote $\mathbf{1} = (1, \ldots, 1) \in \R^m$.
From the definitions it follows that
\[
\ap_P(h)=3 \v_P(\mathbf{1},h,h),\, \ap_P(h,k)= 3 \v_P(\mathbf{1},h,k)~.
\]

Suppose that $P$ is simple and circumscribed, i.e. the vector of the support numbers of $P$ is a multiple of $\mathbf{1}$.
Then Theorem~\ref{thm:AF1 simple} implies the following.

\begin{corollary}
Let $P$ be a simple circumscribed convex polytope. Then for $h,k\in \tilde{H}(P)$,
$$\ap_P(h,k)^2\geq \ap_P(h)\ap_P(k) $$
and equality occurs if and only there exists 
$x\in \R^ 3$ and $\lambda\in \R$ with 
$h=h^x+\lambda k$.
\end{corollary}

Exactly as in the polygon case in Section~\ref{polygons}, it follows that $\ap_P$ has signature $(1,3,m-4)$. 
To remove the assumptions that $P$ is circumscribed, we will need a more general results on convex bodies.
More precisely,  Theorem~\ref{thm: cas eglaite eucl} in the next section implies
the following.

\begin{theorem}\label{thm poly r3}
Let $P$ be a convex polytope. Then for $h,k\in \tilde{H}(P)$,
$$\ap_P(h,k)^2\geq \ap_P(h)\ap_P(k) $$
and equality occurs if and only if there exist 
$x\in \R^ 3$ and $\lambda\in \R$ with 
$h=h^x+\lambda k$.
\end{theorem}

\begin{corollary}
Let $P$ be a convex polytope whose induced metric on the boundary belongs to  $\mathcal{M}(\Sph^2, \alpha)$. Then $H_1(P)$ is 
a convex hyperbolic polyhedron of dimension $\leq m-4$, which embeds isometrically into $\mathcal{M}(\Sph^2, \alpha)$.  
\end{corollary}

\begin{remark}{\rm
For a convex polytope $P$, there may be a convex polytope $P'$ with faces parallel to those of $P$, but with different combinatorics.
Sometimes, as in the case of polygons, 
$H_1(P)$ can be isometrically glued to $H_1(P')$ along a codimension $1$ face.
But, conversely to the polygon case, they would embed into different spaces of flat metrics.
This is because the boundary metrics of $P$ and $P'$ usually have different cone angles, see Figure~\ref{somme angle}.
Note also that if $Q\in \tilde{H}(P')$, then the support vector of $P+Q$ is not $h_P+h_Q$, see \cite{FI} for more details.
}\end{remark}

\begin{figure}[ht]
\begin{center}
\psfrag{a}{$1+2$}
\psfrag{b}{$7+8$}
\psfrag{c}{$1+2$}
\psfrag{d}{$7+8$}
\psfrag{e}{$3+4$}
\psfrag{f}{$5+6$}
\psfrag{x}{$1$}
\psfrag{y}{$2$}
\psfrag{z}{$3$}
\psfrag{t}{$4$} 
\psfrag{u}{$5$}
\psfrag{v}{$6$}
\psfrag{g}{$7$}
\psfrag{h}{$8$}
\psfrag{p}{$3+4$}
\psfrag{q}{$5+6$}
\includegraphics[width=0.6\linewidth]{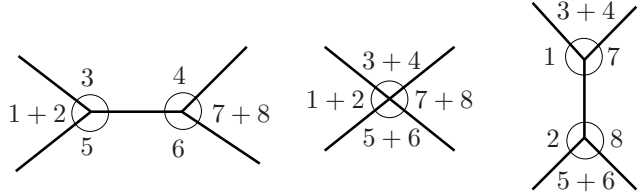}
\end{center}
\caption{Polyhedra with parallel faces but different combinatorics have different cone angles.}\label{somme angle}
\end{figure}

%\begin{figure}[ht]
%\begin{center}
%\includegraphics[scale=1]{sommeangle.png}
%\end{center}
%\caption{Different combinatorics with same face angles leads to different cone angles.}\label{somme angle}
%\end{figure}

\begin{example}
{\rm Let $P$ be a cylinder with a convex polygon $P_0$ as the base. Then $H_1(P)$ is a 
hyperbolic pyramid with an ideal apex and with $H_1(P_0)$ as the base.
}\end{example}

See \cite{FI} for more details on the above examples, as well as for other examples.

\section{Boundary area of convex bodies}

Let $\mathcal{K}_0^3$ be space of convex bodies of $\R^3$, i.e. convex compact sets with non-empty interior.
The volume $\v(K)$ of $K\in \mathcal{K}_0^3$
if the Lebesgue measure of $K$. The \emph{mixed volume} can be defined as
\begin{multline}
\label{defvom}
6\v(K_1,K_2,K_3)=
\v(K_1+K_2+K_3) + \v(K_1)+\v(K_2)+\v(K_3)\\
-\v(K_1+K_2) - \v(K_2+K_3)-\v(K_1+K_3)~.
\end{multline}
If $K_1, K_2$, and $K_3$ are convex polytopes with parallel faces and the same combinatorics,
then above formula corresponds to the polarization of a homogeneous polynomial $\v_P$.
General $K_1,K_2,K_3$ can be approximated by sequences of simple convex polytopes with parallel faces and the same combinatorics.
By continuity, Theorem~\ref{thm:AF1 simple} implies the following.

\begin{theorem}[Alexandrov--Fenchel inequality]\label{AF bodies euc}
For $K_1,K_3,K_3\in \mathcal{K}^3_0$,
$$\v(K_1,K_2,K_3)^2\geq \v(K_1,K_1,K_3)\v(K_2,K_2,K_3)~. $$
\end{theorem}

The complete description of the equality case is still unkown. However we will need only a particular case.

A convex body $K\in  \mathcal{K}^3_0$ is uniquely determined by its support function
$$h_K(v)=\sup_K \langle v,x\rangle $$
defined on $\mathbb{S}^2$.
The Minkowski addition corresponds to the addition of the support functions.
Therefore the mixed volume can be extended by multilinearity to differences of support functions of convex bodies.
A function on $\mathbb{S}^2$ is the support function of a convex body if and only if its homogeneous extension to $\R^3$ is a convex function.
It follows that any $C^2$ functions is the difference of two $C^2$ support functions.
By continuity, one can extend the definition
of $\v(\cdot,K_2,K_3)$ to $C^0(\mathbb{S}^2)$. By the Riesz representation theorem, there exists a Radon measure $\mu(K_2,K_3)$ on $\mathbb{S}^2$ such that

$$\v(h,K_2,K_3)=\frac{1}{3}\int_{\mathbb{S}^2} h \operatorname{d}\mu(K_2,K_3)~. $$

In particular, $\mu(K,K)$ is the \emph{area measure} of $K$, and 
$$\v(K)=\frac{1}{3}\int_{\mathbb{S}^2} h_K \operatorname{d}\mu(K,K)$$
and
$$\ap(K)=\int_{\mathbb{S}^2}\operatorname{d}\mu(K,K) =3\v(K,K,\bold{B}) $$
where $\bold{B}$ is the closed unit ball, with the support function identically equal to $1$.
The mixed-area is then
$$\ap(K_1,K_2)=3\v(K_1,K_2,\bold{B}) $$
and Alexandrov--Fenchel inequality gives
$$\ap(K_1,K_2) \geq \ap(K_1)\ap(K_2)~. $$

By symmetry of the mixed volume, $\v(K,K,\bold{B})=\v(K,\bold{B},K)$, that gives
$$ \ap(K)= \int_{\mathbb{S}^2}h_K\operatorname{d}\mu(K,\bold{B})~.$$

The classical Christoffel problem asks for existence and uniqueness of a convex body $K$
such that, for a prescribed measure $\mu$ we have $\mu=\mu(K,\bold{B})$.
The following theorem provides the uniqueness.
It is an easy consequence of the properties of the Laplacian on the sphere, see Theorem~8.3.6 in \cite{schneider}.

\begin{theorem}\label{them: christ}
We have $\mu(K_1,\bold{B})=\mu(K_2,\bold{B}) $ if and only if  there exits $x\in \R^3$ such that
$K_1=\{x\}+ K_2$.
\end{theorem}

\begin{remark}{\rm
If $P$ is a convex polytope, then $\mu(P,\bold{B})$
is the measure on $\mathbb{S}^2$ supported on the images of the edges of $P$ by the Gauss map, weighted by the  lengths of the corresponding edges of $P$.
}\end{remark}

Let us combine Theorem~\ref{them: christ} with the following algebraic lemma.
The proof is reproduced from \cite[Lemma 7.4.1]{schneider}; we will need it in one of the next sections.

\begin{lemma}\label{lem:algebre}
Let $K_1,K_3\in \mathcal{K}^3_0$. If
$$\ap(K_1,K_3)^2=\ap(K_1)\ap(K_3) $$
then for any $K_2\in \mathcal{K}^3_0$,
$$\ap(K_1)\ap(K_2,K_3)-\ap(K_1,K_2)\ap(K_1,K_3)=0~. $$
\end{lemma}
\begin{proof}
For every $\lambda \geq 0$ by the Alexandrov--Fenchel inequality we have
\[
\ap(K_2+\lambda K_3,K_1)^2-\ap(K_2+\lambda K_3,K_2+\lambda K_3)\ap(K_1)\geq 0~.
\]
This can be rewritten as
\begin{equation}
\label{second 1}
A \lambda^2 - 2\lambda B + C \geq 0 \quad \text{ for all } \lambda > 0~.
\end{equation}
with 
$$A=\ap(K_1,K_3)^2-\ap(K_1)\ap(K_3) $$
$$B=\ap(K_2,K_3)\ap(K_1)-\ap(K_1,K_2)\ap(K_1,K_3) $$
$$C=\ap(K_1,K_2)^2-\ap(K_1)\ap(K_2)~. $$

On the other hand, for $t>0$, Alexandrov--Fenchel inequality gives
\[
\ap(K_2+t\lambda K_1,K_3+tK_1)^2-\ap(K_2+t\lambda K_1)\ap(K_3+tK_1) \geq 0~.
\]
Extracting the coefficient at $t^2$, we obtain
\begin{equation}
\label{second 2}
A \lambda^2 + 2\lambda B + C \geq 0 \quad \text{ for all } \lambda > 0~.
\end{equation}

Together, \eqref{second 1} and \eqref{second 2} imply $B^2-AC \leq 0$.
By assumption $A=0$, therefore $B=0$.
\end{proof}

Thus if
$$\ap(K_1,K_3)^2=\ap(K_1)\ap(K_3)~, $$
then for $\alpha = \ap(K_1,K_3)/\ap(K_1)$ and for any $K_2\in \mathcal{K}^3_0$ we have
$$\ap(K_2,K_3)=\ap(K_2,\alpha K_1)~. $$
Since $K_2$ is abitrary, by linearity and continuity we obtain 
$$\mu(K_3,\bold{B})=\mu(\alpha K_1,\bold{B})~. $$
By invoking Theorem~\ref{them: christ}, we obtain a description of the equality in Theorem~\ref{AF bodies euc} in a special case.

\begin{theorem}\label{thm: cas eglaite eucl}
For $K_1,K_2\in \mathcal{K}_0^3$,
$$\ap(K_1,K_2)^2\geq \ap(K_1)\ap(K_2) $$ with equality if and only if there is $\alpha>0$ and $x\in \R^3$ such that $K_1=\{x\}+\alpha K_2$.
\end{theorem}

This is a reformulation of a theorem of Favard and Kubota, see Theorem~7.6.2 \cite{schneider}.

The argument above has the following heuristic meaning.
The area measure is the gradient of the volume, in the sense that $\int_{\mathbb{S}^2} h \operatorname{d} \mu(K,K)$ is the directional derivative of the volume at $K$ in the direction $h$, for $h\in C^0(\mathbb{S}^2)$ \cite{carlier}.
In the polyhedral case, this is expressed by the fact that
$\partial_i \v(P)=\a(F_i(P))$.

Also, in the same vein, one can consider $\mu(K,\bold{B})$ as a derivative of $\mu(K,K)$. In the polyhedral case,  if $i\not= j$, $\partial_{ij} \v (P)=\frac{1}{\sin \phi_{ij}}l_{ij}$, where $l_{ij}$ is the length of the edge between the $i$h face and the $j$th face (considered as $0$ if the faces do not meet along an edge).  Alexandrov--Fenchel inequality says that the map $$(K_1,K_2)\mapsto
\ap(K_1,K_2)^2-\ap(K_1)\ap(K_2)$$ takes only non-positive values.
Thus, the zeros of this map are its critical points. The vanishing of the partial derivatives leads to the fact that 
 $\mu(K_1,\bold{B})$ and $\mu(K_2,\bold{B})$ are proportional.

\section{Alternative argument for the signature of the area form}
\label{sobolev}
Here we outline an alternative argument for Theorem~\ref{thm: cas eglaite eucl}.
The area of a convex body $K$ with support function $h_K:\mathbb{S}^2\to \R$ can be computed as
(see \cite[Theorem~4.2a]{heil}, \cite[p.297,298]{schneider})

\begin{equation}\label{eq:a lap}\ap(K)= \int_{\mathbb{S}^2} h_K^2 - \frac{1}{2}\|\nabla h_K\|^2. \end{equation}
Here $\nabla h_K$ is the weak spherical gradient of $h_K$, and $\|\cdot\|$ is its norm with respect to the metric of the unit sphere.
Note that this formula is related to the spherical Laplacian $\Delta$ if $h_K$ is $C^2$. Namely, using Green formula one obtains
$$\ap(K)=\int_{\mathbb{S}^2} h_K^2 + \frac{1}{2} h_K\Delta h_K~. $$

Recall that $2$ is the first non-zero eigenvalue $\lambda_1$ of the spherical Laplacian.
The associated eigenspace is the space of restriction to $\mathbb{S}^2$ of linear functions on $\R^3$.
But linear functions are exactly the support functions of points.
Thus quotienting by the eigenspace of $\lambda_1$ is the same as considering convex bodies up to translations. 

Also, the eigenspace of the spherical laplacian for the zero eigenvalue is the space of constant functions on $\mathbb{S}^2$. Using the  Rayleigh characterization of the first
nonzero eigenvalue, it is easy to deduce that, on any finite dimensional vector space spanned by support functions of convex bodies, quotiented by linear forms, $\ap$ has Lorentzian signature: it is non-degenerate with a unique positive direction.
Then Theorem~\ref{thm poly r3} is simply a translation of the
reversed Cauchy--Schwarz inequality for Lorentzian vector spaces.
We refer to \cite{DF} for details about this argument and generalization to higher dimensions. The above  argument is a generalization to the dimension $3$ of the content of Remark~\ref{remark debin}.

\section{Fuchsian convex polyhedra}\label{fuchsian}
The Minkowski space $\R^{2,1}$ is $\R^3$ equipped with the symmetric bilinear form
$$\langle x,y \rangle_{2,1}=x_1y_1+x_2y_2-x_3y_3~. $$
For $r>0$ denote
$$\H^2_r=\{x | \langle x,x\rangle_{2,1}=-r^2, x_3>0 \}~.$$ 
The family $\H^2_r$, $r>0$, forms a foliation of the future cone of the origin $I^+(0)=\{x | \langle x,x\rangle_{2,1}<0 \}$.
Note that $\H^2_1:=\H^2$ equipped with the  metric induced from the Minkowski scalar product is a model of the hyperbolic plane.

A plane $F$ in $\R^{2,1}$ is called \emph{space-like} if the
restriction of $\langle \cdot,\cdot \rangle_{1,2}$ to it is positive definite. Its \emph{future side} is the half-space bounded by $F$ which contains the direction $(0,0,1)$. 

Let $F_1,\ldots,F_m$ be space-like planes, $F_i$ tangent to $\H^2_{r_i}$ and  let $\Gamma$ be a group of linear isometries of $\R^{2,1}$, such that $\H^2/\Gamma$ is a compact (oriented) hyperbolic surface. 

A \emph{convex $\Gamma$-polyhedron} $P$ is the intersection of the future side of the orbits of the $F_i$ for $\Gamma$.
It is easy to see that there exist $r,r'$ such that $\H^2_r \subset P \subset \H^2_{r'}$, and that the faces are compact convex polygons.

The induced metric on the boundary of $P$ is a flat metric with conical singularities.
Since the image of the  Gauss map of a vertex is a convex polygon in the hyperbolic space,
the singular curvatures at the vertices are negative (the argument is similar to the one of Figure~\ref{fig gauss}).
Thus, the induced metric on 
$\partial P/\Gamma$ is a flat metric with cone singularities of negative curvature.

Denote
%$$H^\Gamma(P)=\{\text{ convex } \Gamma\text{-polyhedra with parallel faces and same combinatorics as }P\}~. $$

\begin{multline*}
H^\Gamma(P)=\{\text{ convex } \Gamma\text{-polyhedra with faces parallel to those of }P\\
\text{ and the same combinatorics as }P\}~.
\end{multline*}

As in the case of convex polytopes, one can check that
 convex $\Gamma$-polyhedra are determined by their support numbers $h_1,\ldots,h_m$ and that
$H^\Gamma(P)$ can be identified with a convex polyhedral cone in  
$\R^m$. This cone has a non-empty interior if and only if $P$ is simple.

Similarly to Figure~\ref{fig sup edge} and Figure~\ref{fig supface}, a support number of a face is computed as (\cite[Lemma~2.2]{Fem}):
\begin{equation}\label{eq: supp num mink}
 h_{ij}=-\frac{h(j)-h(i)\cosh \varphi_{ij}}{\sinh \varphi_{ij}}~,
\end{equation}
in particular, edge lengths are linear functions of the support numbers. See also Figure~\ref{fig fuch}.

\begin{figure}[ht]
\begin{center}
\psfrag{pi}{$p_i$}
\psfrag{pik}{$p_{ik}$}
\psfrag{hijk}{$h_{ijk}$}
\psfrag{hji}{$h_{ji}$}
\psfrag{hij}{$h_{ij}$}
\psfrag{om}{$\omega_{ijk}$}
\psfrag{0}{$0$}
\psfrag{hi}{$h_i$}
\psfrag{hj}{$h_j$}\psfrag{hik}{$h_{ik}$}
\psfrag{phi}{$\phi_{ij}$}
\includegraphics[width=0.7\linewidth]{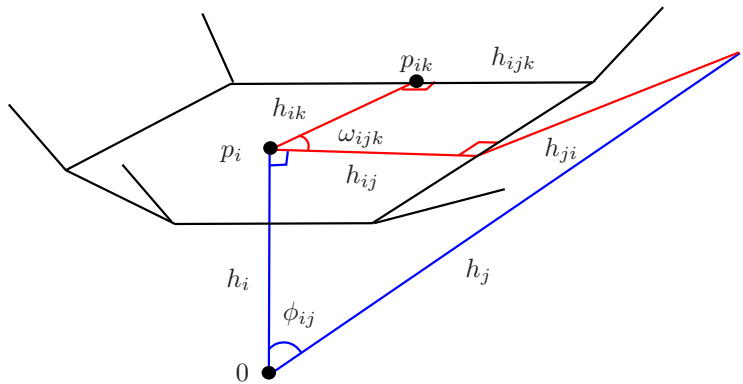}\caption{$h_{ij}=-\frac{1}{\sinh \varphi_{ij}}(h_j-h_i\cosh\varphi_{ij})$, $h_{ijk}=\frac{1}{\sin \omega_{ijk}}(h_{ij}-h_{ik}\cos\omega_{ijk})$}\label{fig fuch}
\end{center}
\end{figure}

Note that for any convex $\Gamma$-polyhedra $P_1,P_2$ and any $\lambda > 0$ the Minkowski sum
$P_1+\lambda P_2$ is a convex $\Gamma$-polyhedron.
The \emph{covolume} of $P$ is defined as the volume (with respect to the Lebesgue measure of $\R^3$) of 
$\left( I^+(0) \setminus P\right)/\Gamma$.
Note that 
$$ (I^+(0) \setminus P_1) + (I^+(0) \setminus P_2)\not=  I^+(0) \setminus (P_1 +  P_2)~. $$

The covolume can be computed as
$$\cov(P)=\frac{1}{3}\sum_{i=1}^m h_i \a(F_i) $$
where $\a(F_i)$ is the area of the $i$th face for the induced metric on the space-like plane containing $F_i$. 
Note that if instead of the Minkowski scalar product we consider the Euclidean scalar product,
then the covolume does not change its value, while the support numbers and face areas do.

On $\R^m$, the covolume defines a degree $3$ homogeneous polynomial 
$$\cov_P(h)=\frac{1}{3}\sum_{i=1}^m h_i \a_{F_i}(h_{i\bullet})$$ 
where $h_{i\bullet}$ is the support vector of the $i$th face.
The \emph{mixed covolume} is the symmetric $3$-linear form
$$\cov_P(h,k,p)=\frac{1}{3}\sum_{i=1}^m h_i \a_{F_i}(k_{i\bullet}, p_{i\bullet})~. $$

We have the following result \cite{Fgafa} that emphasizes an important difference between the Euclidean case and the Minkowski case.
Contrarily to the Euclidean case, the proof is a direct calculation; a similar calculation does not suffice in the Euclidean case.
A naive reason is that 
$\cos$ is replaced by $\cosh$ in Lorentzian geometry, and that $\cosh x \ge 1$.
A deeper reason will be explained in the next sections.

\begin{theorem}\label{thm:cov def pos}
For $h\in H^\Gamma(P)$,
$\cov_P(\cdot,\cdot,h)$ is positive definite.

\end{theorem}
\begin{proof}
By $3$-linearity, $D^2_{k,p}\cov_P(h)=3\cov_P(h,k,p)$.
Hence it suffices to prove that the Hessian  of $\cov_P$ at $h$ is positive definite. One can show that
$$D^2_{k,p}\cov_P(h) = \sum_{i=1}^m p_i D_k \a_{F_i} (h_{i\bullet}) $$
hence it suffices to study the Jacobian matrix of the map from $\R^m$ to $\R^m$ that associated to $h$ the areas of the faces.
A straightforward computation yields the result: from \eqref{eq: supp num mink},
\begin{eqnarray}
 \ \label{eq:der lor1} &&\frac{\partial h_{ij}}{\partial h_j}=-\frac{1}{\sinh \varphi_{ij}}~, \\ 
\ &&\frac{\partial h_{ij}}{\partial h_i}=\frac{\cosh \varphi_{ij}}{\sinh \varphi_{ij}}~.\label{eq:der lor2}
\end{eqnarray}
If $h_i=h_j$ and if the quadrilateral is deformed under this condition, then
\begin{equation*}
 \frac{\partial h_{ij}}{\partial h_i}=\frac{\cosh \varphi_{ij}-1}{\sinh \varphi_{ij}}~.\label{eq:der lor3}
\end{equation*}

Recall that for Euclidean polygons we have
\begin{equation}\label{eq:der areab}
\frac{\partial \a(F_i)}{\partial h_{ik}}=l_{ik}~,
\end{equation}
where $l_{ik}$ is the length of the edge between the face supported by $h_i$ and the one supported by $h_k$.

We denote by $E_i^j\subset\Gamma\mathcal{I}$ is the set of indices $k\in\Gamma j$ such that $F_k$ 
 is adjacent to $F_i$ along an edge. 
If $j\in\mathcal{I}\setminus\{i\}$ we get
$$\frac{\partial \a(F_i)}{\partial h_j}=\sum_{k\in E_i^j}\frac{\partial \a(F_i)}{\partial h_{ik}}\frac{\partial h_{ik}}{\partial h_j}~.$$
From \eqref{eq:der lor1} and \eqref{eq:der areab} it follows that 
\begin{equation}\label{eq: der par A}
\frac{\partial \a(F_i)}{\partial h_j}=-\sum_{k\in E_i^j}\frac{l_{ik}}{\sinh \varphi_{ik}}~.
\end{equation}
For the diagonal terms:
\begin{eqnarray}\label{eq: coefdiag}
\ \frac{\partial \a(F_i)}{\partial h_i}&=&\sum_{j\in\mathcal{I}\setminus\{i\}}\sum_{k\in E_i^j}
\frac{\partial \a(F_i)}{\partial h_{ik}}\frac{\partial h_{ik}}{\partial h_i}
+\sum_{k\in E_i^i}
 \frac{\partial \a(F_i)}{\partial h_{ik}}\frac{\partial h_{ik}}{\partial h_i}\nonumber \\
\ &\stackrel{(\ref{eq:der areab},\ref{eq:der lor2},\ref{eq:der lor3})}{=}&\sum_{j\in\mathcal{I}\setminus\{i\}}
\sum_{k\in E_i^j}\cosh \varphi_{ik}\frac{ l_{ik}}{\sinh \varphi_{ik}} +
 \sum_{k\in E_i^i} l_{ik}\frac{\cosh \varphi_{ik}-1}{\sinh \varphi_{ik}}~.
\end{eqnarray} 

Since $\cosh x>1$ for $x\not= 0$, \eqref{eq: coefdiag} and \eqref{eq: der par A} lead to
$$\frac{\partial \a(F_i)}{\partial h_i}  > \sum_{j\in\mathcal{I}\setminus\{i\}} \left|\frac{\partial \a(F_i)}{\partial h_j} \right|>0$$
that means that the Jacobian is symmetric, strictly diagonally dominant
with positive diagonal entries, hence positive definite.
 
\end{proof}

The boundary area of $P$ is the sum of the areas of its faces
(recall that the plane spanned by a face carries a Euclidean metric).
This defines the following quadratic form on $\R^m$:
$$\am_P(h)=\sum_{i=1}^m \a_{F_i}(h(i))~. $$

Suppose that $P$ is simple and circumscribed about the hyperboloid, that is has the support vector $\mathbf{1}=(1,\ldots,1) \in \R^m$.
Then Theorem~\ref{thm:cov def pos} with $h=\mathbf{1}$ implies the following.

\begin{corollary}
Let $P$ be a simple circumscribed convex $ \Gamma$-polyhedron. Then $\am_P$ is positive definite.
\end{corollary}

Let  $H_1^\Gamma(P)$ be the quotient of $H^\Gamma(P)$ by homotheties. Hence, if $P$ is simple and circumscribed, $H_1^\Gamma(P)$ is 
a convex spherical polyhedron (with non empty interior) of $\mathbb{S}^{m-1}$.  
To remove the assumptions on $P$, we will need a more general results on $\Gamma$ convex sets. Namely,  Theorem~\ref{thm: cas eglaite mink} in the next section implies the following result.

\begin{theorem}\label{thm prin fuch}
Let $P$ be a convex $\Gamma$-polyhedron. Then the quadratic form $\am_P$ is positive definite.
\end{theorem}

Similarly to the convex body case, we deduce the following.

\begin{corollary}
Let $P$ be a convex $\Gamma$-polyhedron, such that the induced metric on $\partial P/\Gamma$  belongs to  $\mathcal{M}(S, \alpha)$. Then $H_1(P)$ is 
a convex spherical polyhedron of dimension $\leq m-1$, which embeds isometrically into a Veech leaf of $\mathcal{M}(S, \alpha)$. 
\end{corollary}

\begin{remark}{\rm
There is a similar construction in the Minkowski plane.
Formally in the same way than   the case of convex compact polygons (see Remark~\ref{remark polygon}),
one obtains spherical orthoschemes \cite{Fem}. But they are not related to spaces of flat metrics.
}\end{remark}

\section{Area form on Fuchsian convex bodies}

Let $\mathcal{K}_\Gamma^3$ be space of convex $\Gamma$-sets of $\R^3$, i.e. convex  sets contained in $I^+(0)$ and invariant under the action of $\Gamma$ (it can be checked that such a convex set has only space-like support planes). 
The covolume is defined as in the polyhedral case. The \emph{mixed covolume} is
defined similar to the mixed volume  in \eqref{defvom}.
We will to prove the following.

\begin{theorem}\label{thm: cas eglaite mink}
For $K_1,K_2\in \mathcal{K}_\Gamma^3$,
$$\ap(K_1,K_2)^2\leq \ap(K_1)\ap(K_2) $$ with equality if and only if there is $\alpha>0$  such that $K_1=\alpha K_2$.
\end{theorem}
Below three proofs from different points of view are given.

\subsection{First argument: mimicking the convex bodies case}

As $K_1,K_2,K_3$ can be approximated by a sequence of three simple convex $\Gamma$-polyhedra with parallel faces and same combinatorics,
Theorem~\ref{thm:cov def pos}, Cauchy--Schwarz inequality and a limit argument give the following (see \cite{Fgafa} for details).

\begin{theorem}[Reversed Alexandrov--Fenchel inequality]\label{AF bodies mink}
For $K_1,K_3,K_3\in \mathcal{K}^3_\Gamma$,
$$\cov(K_1,K_2,K_3)^2\leq \cov(K_1,K_1,K_3)\cov(K_2,K_2,K_3)~. $$
\end{theorem}

Identify $K\in  \mathcal{K}^3_\Gamma$ with its support function
$$h_K(v)=\sup_K \langle v,x\rangle_{2,1} $$
which is $\Gamma$-invariant. We consider $h_K$ as a function on the compact surface
$\H^2/\Gamma$. One can also prove that there exists a Radon measure $\mu(K_2,K_3)$ on $\mathbb{H}^2/\Gamma$ such that \cite[Remark~3.15]{fv}

$$\cov(h,K_2,K_3)=-\frac{1}{3}\int_{\mathbb{H}^2/\Gamma} h \operatorname{d}\mu(K_2,K_3)~. $$

In particular, $\mu(K,K)$ is the area measure of $K$, and 
$$\cov(K)=-\frac{1}{3}\int_{\mathbb{H}^2/\Gamma} h_K \operatorname{d}\mu(K)$$
and
$$\am(K)=\int_{\mathbb{H}^2/\Gamma}\operatorname{d}\mu(K,K) =3\v(K,K,\bold{H}) $$
where $\bold{H}$ is the convex $\Gamma$-set bounded by $\H^2$.

Reversed Alexandrov--Fenchel inequality (Theorem~\ref{AF bodies mink}) then reads as
follows:
for $K_1,K_2\in \mathcal{K}_\Gamma^3$,
$$\am(K_1,K_2)^2\leq \am(K_1)\am(K_2)~. $$

But by symmetry of the mixed volume,
$$ \am(K)= \int_{\mathbb{H}^2/\Gamma}h_K\operatorname{d}\mu(K,\bold{H})~.$$

The $\Gamma$-Christoffel problem asks for existence and uniqueness of a  convex $\Gamma$-set $K$
such that, for a prescribed Radon measure $\mu$ on $\mathbb{H}^2/\Gamma$, $\mu=\mu(K,\bold{H})$. The uniqueness result is as follows.
It is a simple consequence of properties of the Laplacian on a compact hyperbolic manifold.

\begin{theorem}[{\cite{fv}}]\label{them: christ mink}
We have $\mu(K_1,\bold{H})=\mu(K_2,\bold{H}) $ if and only if  
$K_1= K_2$.
\end{theorem}

The proof of the following lemma is the same as the proof of Lemma~\ref{lem:algebre}.

\begin{lemma}\label{lem:algebre fuc}
Let $K_1,K_3\in \mathcal{K}_\Gamma^3$. If
$$\am(K_1,K_3)^2=\am(K_1)\am(K_3) $$
then for any $K_2\in \mathcal{K}^3_\Gamma$,
$$\am(K_1)\am(K_2,K_3)-\am(K_1,K_2)\am(K_1,K_3)=0~. $$
\end{lemma}

So, exactly as in the Euclidean case, if equality occurs in the reversed Alexandrov--Fenchel inquality, there is a positive $\alpha$ such that 
$$\mu(K_3,\bold{H})=\mu(\alpha K_1,\bold{H}) $$
and by Theorem~\ref{them: christ mink}  and Theorem~\ref{AF bodies mink}, Theorem~\ref{thm: cas eglaite mink} is proved.

\subsection{Second argument: explicit formula for the area}

Here the argument is similar to that Section~\ref{sobolev}, but much simpler.
If $h_K$ is the support function of a convex $\Gamma$-set $K$, then
\begin{equation}\label{eq:aire lor}\am(K)= \int_{\mathbb{H}^2/\Gamma} h_K^2 + \frac{1}{2}\|\nabla h_K\|^2 \end{equation}
where $\nabla$ is the weak hyperbolic gradient, and $\|\cdot\|$ is the norm given by the hyperbolic metric. This form is obviously positive-definite on any vector space spanned by support functions of $\Gamma$-convex sets.

Equation \eqref{eq:aire lor} can be proved similarly to \eqref{eq:a lap}, first proving the formula in the regular case, then use an approximation argument (see Appendix~A in \cite{bf} for smooth approximation of $\Gamma$-convex sets). Details are left to the interested reader.

\subsection{Relations with the  covolume}\label{sec covolume}

There is no need of an ambient Lorentzian metric to state the reversed Alexandrov--Fenchel inequlity (Theorem~\ref{thm: cas eglaite mink}).
Actually, a fundamental domain for the action of $\Gamma$ can be given by a convex cone $C$, contained in the cone
$I^+(0)$. For example, one can consider a Dirichlet fundamental convex polyhedron in $\H^2\subset I^+(0)$ for the action of $\Gamma$, which defines a convex cone $C$ in $I^+(0)$. The covolume of a convex $\Gamma$-polyhedron is the volume of $C\setminus (K\cap C)$. 
Once the convex cone $C$ is fixed, $K\cap C$ is a \emph{$C$-coconvex body}.
More precisely, a conconvex body (in $C$) is a convex set $A\subset C$ 
such that $C\setminus A$ is bounded and not empty.

A reversed  Alexandrov--Fenchel inequality holds for $C$-coconvex bodies, it is actually deduced from the classical Alexandrov--Fenchel inequality for convex bodies, see
\cite{KT}. The description of the equality case is still unknown.

\begin{question}\label{question covol}
Does equality occur in reversed Alexandrov--Fenchel inequality for $C$-coconvex bodies if and only if $K_1=\lambda K_2$ for some $\lambda>0$?
\end{question}

A positive answer to Question~\ref{question covol} would imply Theorem~\ref{thm: cas eglaite mink} as a very particular case.

We know that the covolume is strictly convex \cite{Fco}. On the one hand, this supports the idea that Question~\ref{question covol} has a positive answer. On the other hand,  the positive definiteness of the Hessian of the covolume (in a suitable sense) is related to 
Theorem~\ref{thm: cas eglaite mink} (see the beginning of the proof of Theorem~\ref{thm:cov def pos}).

%
%
%\begin{figure}[ht]
%\begin{center}
%%\includegraphics[scale=1]{covol.png}
%\end{center}
%\caption{Let $C$ be truncated cone which is the intersection of a (convex) fundamental domain for $\Gamma$  with a half-space bounded by $P$. Let $\bar K$ be the intersection of $C$ with $K$. Then $\cov(K)=\v(C)-\v(\bar K)$. $\bar K$ is a kind of convex body known as \emph{convex cap}. A simple application of Fubuni theorem gives that volume on the set of convex caps with given basis is strictly concave. Hence the covolume is strictly convex. See \cite{bf} for more details.}\label{fig covol}
%\end{figure}

\section{Acknowledgement}

This note is based on remarks 4.3 and 4.4 in \cite{FI}, see also \cite{Fgafa,Fem}. 
Theorem~\ref{thm poly r3} of the present paper is Corollary~2.9 in 
\cite{FI}. Theorem~\ref{thm:cov def pos} comes from \cite{Fgafa}. Theorem~\ref{thm prin fuch} of the present paper is new.
The subject of \cite{FI} is to equip different spaces of convex polytopes with parallel faces (in any dimension) with hyperbolic structures using the mixed volume.

The authors thank Duc Manh Nguyen for useful conversations.

\begin{spacing}{0.9}
\begin{footnotesize}
\bibliography{flat}
\bibliographystyle{alpha}
\end{footnotesize}
\end{spacing}

\end{document}